\documentclass[reqno,12pt]{amsart}
\usepackage{a4wide, color, eucal, enumerate, mathrsfs, xcolor}
\usepackage{amsmath, amssymb, epsfig, amsthm}
\usepackage{mathtools}
\usepackage{lipsum}
\usepackage[english]{babel} 
\usepackage[T1]{fontenc}
\usepackage[utf8]{inputenc}
\usepackage{hyperref}
\hypersetup{hidelinks, urlcolor=blue, pdfstartview=FitH, pdfpagemode=UseOutlines, pdfnewwindow, breaklinks}
\PassOptionsToPackage{unicode}{hyperref}
\usepackage{bookmark, cleveref, float, multirow}

\definecolor{myRed}{HTML}{da4939}
\let\subset\subseteq

\let\phi\varphi
\let\epsilon\varepsilon
\renewcommand{\:}{\colon}
\newcommand{\smallbullet}{\raisebox{0.15em}{\scalebox{0.5}{$\bullet$}}}
\newcommand{\wto}{\rightharpoonup}

\newcommand{\Id}{\operatorname{Id}}
\newcommand{\N}{\mathbb{N}}
\newcommand{\R}{\mathbb{R}}
\newcommand{\Q}{\mathbb{Q}}
\newcommand{\Z}{\mathbb{Z}}
\providecommand{\qm}[1]{``#1''}
\newcommand{\LN}{($N$) } 
\newcommand{\lin}{\operatorname{\pounds}}
\DeclareMathOperator{\dist}{dist}
\newcommand{\dx}{\, \mathrm{d}x}
\newcommand{\dt}{\, \mathrm{d}t}

\newcommand{\INV}{(INV) }

\newcommand{\step}[2]{\vspace{3mm}%
\belowpdfbookmark{\the\value{section}.#1 #2}{\the\value{section}.#1}%
\noindent \textbf{Step #1:} \textsl{#2}}

\numberwithin{equation}{section}
\counterwithout{figure}{section}
\newtheorem{theorem}{Theorem}[section]

\newtheorem{lemma}[theorem]{Lemma}

\newtheorem{definition}[theorem]{Definition}

\theoremstyle{remark}

\makeatletter
\newcommand*{\mint}[1]{%
  \mint@l{#1}{}%
}
\newcommand*{\mint@l}[2]{%
  \@ifnextchar\limits{%
    \mint@l{#1}%
  }{%
    \@ifnextchar\nolimits{%
      \mint@l{#1}%
    }{%
      \@ifnextchar\displaylimits{%
        \mint@l{#1}%
      }{%
        \mint@s{#2}{#1}%
      }%
    }%
  }%
}
\newcommand*{\mint@s}[2]{%
  \@ifnextchar_{%
    \mint@sub{#1}{#2}%
  }{%
    \@ifnextchar^{%
      \mint@sup{#1}{#2}%
    }{%
      \mint@{#1}{#2}{}{}%
    }%
  }%
}
\def\mint@sub#1#2_#3{%
  \@ifnextchar^{%
    \mint@sub@sup{#1}{#2}{#3}%
  }{%
    \mint@{#1}{#2}{#3}{}%
  }%
}
\def\mint@sup#1#2^#3{%
  \@ifnextchar_{%
    \mint@sup@sub{#1}{#2}{#3}%
  }{%
    \mint@{#1}{#2}{}{#3}%
  }%
}
\def\mint@sub@sup#1#2#3^#4{%
  \mint@{#1}{#2}{#3}{#4}%
}
\def\mint@sup@sub#1#2#3_#4{%
  \mint@{#1}{#2}{#4}{#3}%
}
\newcommand*{\mint@}[4]{%
  \mathop{}%
  \mkern-\thinmuskip
  \mathchoice{%
    \mint@@{#1}{#2}{#3}{#4}%
        \displaystyle\textstyle\scriptstyle
  }{%
    \mint@@{#1}{#2}{#3}{#4}%
        \textstyle\scriptstyle\scriptstyle
  }{%
    \mint@@{#1}{#2}{#3}{#4}%
        \scriptstyle\scriptscriptstyle\scriptscriptstyle
  }{%
    \mint@@{#1}{#2}{#3}{#4}%
        \scriptscriptstyle\scriptscriptstyle\scriptscriptstyle
  }%
  \mkern-\thinmuskip
  \int#1%
  \ifx\\#3\\\else_{#3}\fi
  \ifx\\#4\\\else^{#4}\fi
}
\newcommand*{\mint@@}[7]{%
  \begingroup
    \sbox0{$#5\int\m@th$}%
    \sbox2{$#5\int_{}\m@th$}%
    \dimen2=\wd0 %
    \let\mint@limits=#1\relax
    \ifx\mint@limits\relax
      \sbox4{$#5\int_{\kern1sp}^{\kern1sp}\m@th$}%
      \ifdim\wd4>\wd2 %
        \let\mint@limits=\nolimits
      \else
        \let\mint@limits=\limits
      \fi
    \fi
    \ifx\mint@limits\displaylimits
      \ifx#5\displaystyle
        \let\mint@limits=\limits
      \fi
    \fi
    \ifx\mint@limits\limits
      \sbox0{$#7#3\m@th$}%
      \sbox2{$#7#4\m@th$}%
      \ifdim\wd0>\dimen2 %
        \dimen2=\wd0 %
      \fi
      \ifdim\wd2>\dimen2 %
        \dimen2=\wd2 %
      \fi
    \fi
    \rlap{%
      $#5%
        \vcenter{%
          \hbox to\dimen2{%
            \hss
            $#6{#2}\m@th$%
            \hss
          }%
        }%
      $%
    }%
  \endgroup
}
\makeatother
\newcommand{\aint}{\mint{-}}

\begin{document}
\title{Weak limits of Sobolev homeomorphisms \linebreak are one to one} 

\author{Ondřej Bouchala}
\address{Czech Technical University in Prague, Faculty of Information Technology, Thákurova~9, 160 00 Prague 6, Czech Republic}
 \email{\tt ondrej.bouchala@gmail.com}

\author{Stanislav Hencl}
\address{Department of Mathematical Analysis, Charles University,
So\-ko\-lovsk\'a 83, 186~00 Prague 8, Czech Republic}
\email{\tt hencl@karlin.mff.cuni.cz}

\author{Zheng Zhu}
\address{School of Mathematical Science\\
Beihang University\\   
Changping District Shahe Higher Education Park South Third Street No. 9\\
Beijing 102206\\
P. R. China}
\email{\tt zhzhu@buaa.edu.cn}

\keywords{limits of Sobolev homeomorphisms, topological degree, linking number} 
\thanks{S. Hencl was supported by the grant GA\v CR P201/24-10505S. Z. Zhu was supported by the NSFC grant (No. 12301111) and the starting grant from Beihang University (ZG216S2329). This research was done when Z.\ Zhu was visiting the
 Department of Mathematical Analysis, Faculty of Mathematics and Physics, Charles University. He wishes to thank Charles University for its hospitality.}
\subjclass[2000]{46E35}

\begin{abstract}
	We prove that the key property in models of Nonlinear Elasticity which corresponds to the non-interpenetration of matter, i.e.\ injectivity a.e., can be achieved in the class of weak limits of homeomorphisms under very minimal assumptions. 
 	
 Let $\Omega\subset \R^n$ be a domain and let $p>\left\lfloor\frac{n}{2}\right\rfloor$ for $n\geq 4$ or $p\geq 1$ for $n=2,3$. Assume that $f_k\in W^{1,p}$ is a sequence of homeomorphisms such that $f_k\rightharpoonup f$ weakly in $W^{1,p}$ and assume that $J_f>0$ a.e. Then we show that $f$ is injective a.e.
\end{abstract}

\maketitle

\section{Introduction}

In this paper, we study classes of mappings that might serve as deformations in models of Nonlinear Elasticity. Let $\Omega\subset\R^n$ be a domain that corresponds to our body and let $f\colon\Omega\to\R^n$ be a mapping that describes the deformation of the body. 
Following the pioneering papers of Ball \cite{B1981} and Ciarlet and Ne\v{c}as \cite{CN1987} we ask if our mapping is in some sense injective as the physical \qm{non-interpenetration of matter} asks a deformation to be one-to-one.

Let us recall some known results about the injectivity, or at least injectivity almost everywhere of the mapping $f$.
Following Ball \cite{B1981} we can require that the energy functional $\int_{\Omega} W(Df)$ contains special terms (like ratio of powers of $Df$, $\operatorname{adj} Df$ and $J_f$)
and any mapping with finite energy and reasonable boundary data is a homeomorphism (the reader is referred to e.g.\ \cite{IS1993, VM1998} and \cite{S1988} for related results).

In the results like Ball \cite{B1981}, we need to know that our mapping is continuous everywhere, but in some real-world deformations cavitation or even fractures may occur. To allow such a behavior in our models we do not require injectivity of our mapping, but only injectivity almost everywhere. 
Ciarlet and Ne\v{c}as \cite{CN1987} studied the class of mappings that satisfy
\begin{equation}\label{cond:Ciarlet-Necas}
\int_{\Omega} J_f\leq |f(\Omega)|
\end{equation}
together with $J_f>0$ a.e.\ and they showed that the mappings from this class are injective a.e.\ in the image, see e.g.\ \cite{B2002, BHM2020, MTY1994, T1988} for further results in this direction.
The inequality \eqref{cond:Ciarlet-Necas} is nowadays called the Ciarlet--Nečas condition.
The assumption $J_f>0$ a.e.\ is very natural in models of Nonlinear Elasticity as the \qm{real deformation} cannot change its orientation and we cannot compress our material too much, so the set $\{J_f=0\}$ must be small.

Another approach can be traced to M\"uller and Spector \cite{MS1995}, where they studied a~class of mappings that satisfy $J_f>0$ a.e.\ together with the \INV condition (see e.g.\ \cite{BHM2020,CD2003,HMC2012,MST1996}).
Besides many other interesting results they showed that mappings in their class are injective a.e. Informally speaking, the \INV condition means that the ball $B(x,r)$ is mapped inside the image of the sphere $f(S(a,r))$ and the complement $\Omega\setminus \overline{B(x,r)}$ is mapped outside of $f(S(a,r))$.

In this paper we study injectivity a.e.\ in the class of weak limits of Sobolev homeomorphisms. The natural discontinuous mappings like cavitation belong to this class and this class (with uniformly bounded $\int_{\Omega} |Df|^p$) is weakly closed, which makes it a suitable class for the approach of Calculus of Variations (see e.g.\ \cite{DHM2023} and \cite{DHM2024}). 
Let us note that the class of weak limits of Sobolev homeomorphisms was recently characterized in the planar case by Iwaniec and Onninen \cite{IO2016, IO2017} and De Philippis and Pratelli \cite{DPP2020} (see also Campbell \cite{C2024}). 

Let us recall that weak limits of homeomorphisms in $W^{1,p}$, $p>n-1$, with $J_f>0$ a.e.\ are injective a.e., as they obviously satisfy the \INV condition, the \INV condition is stable under weak convergence and \INV mappings with $J_f>0$ a.e.\ are injective a.e.\ (see~\cite{MS1995}). 
In the recent paper, Bouchala, Hencl and Molchanova proved that \cite[Theorem~1.1]{BHM2020} there is an example of a continuous mapping which is even a strong limit of $W^{1,n-1}$ homeomorphisms, $n\geq 3$, such that there is a set of positive measure $|K|>0$ so that $|f^{-1}(K)|>0$ and $f^{-1}(y)$ is a continuum for every $y\in K$. However, this example does not contradict injectivity a.e.\ as one can divide each $f^{-1}(y)$ into the \qm{endpoint} $e_y$ and the rest $f^{-1}(y)\setminus \{e_y\}$ and the set
$$
\bigcup_{y\in K}f^{-1}(y)\setminus \{e_y\}
$$

has zero measure, so our map is indeed injective outside of this set. Another pathological example of weak limits of $W^{1,n-1}$ homeomorphisms can be found in Conti and De Lellis \cite[Theorem~6.3]{CD2003} (see also \cite[Theorem 1.2]{DHM2023}), where the limit mapping can change orientation on some set and it can fail the \INV condition (and thus it maps something inside of a~ball to outside of the topological image of the corresponding sphere), but it is still injective a.e.  
We wanted to know if this is a general phenomenon and if each weak limit of $W^{1,n-1}$ homeomorphisms must always be injective a.e. Somewhat surprisingly we can prove such a positive result under much weaker Sobolev regularity. 

\begin{theorem}\label{main} 
Let $n\geq 2$, $\Omega\subset\R^n$ be a domain and let $p>\left\lfloor\frac{n}{2}\right\rfloor$ for $n\geq 4$ or $p\geq 1$ for $n=2,3$. Let $f_k\in W^{1,p}(\Omega,\R^n)$ be a sequence of homeomorphism such that $f_k\rightharpoonup f$ weakly in $W^{1,p}$ and assume that $J_f>0$ a.e. 

Then $f$ is injective a.e., i.e.\ there is a set $N\subset \Omega$ with $|N|=0$ such that $f|_{\Omega\setminus N}$ is injective. 
\end{theorem}

Let us note that the assumption $J_f>0$ a.e.\ is necessary as without it we would be e.g.\ able to squash some ball into a point and injectivity a.e.\ would obviously fail. 
To prove this result under such a minimal regularity assumption we use the topological notion of the linking number, which was previously used in the context of Sobolev homeomorphisms and their limits in \cite{HM2010}, \cite{HO2016} and \cite{GH2019}. 
We would like to know if the assumption $p>\left\lfloor\frac{n}{2}\right\rfloor$ is sharp and if there are counterexamples to injectivity a.e.\ below this regularity. 

Let us briefly explain the main idea of the proof in the case $n=3$, which is nicely captured in Fig.\ \ref{fig:LinkedSquares} below. For the contradiction, we assume that there are two sets $A$~and~$B$ of positive measure that are mapped onto the same set. Without the loss of generality, we assume that all points of $A\cup B$ are points of approximative differentiability and Lebesgue points of the derivative. It follows that we can find a green square curve (see Fig.\ \ref{fig:LinkedSquares}) whose big part lies in $A$ and it is mapped to almost the same shape with only some small wiggles on the boundary. Similarly, we find a red square curve (mainly inside of $B$) which is mapped to a similar shape with only small wiggles. Original green and red squares curves are far away and they are not linked (see formal definition of the linking number and Fig.\ \ref{fig:linkingnumber} in Preliminaries) but as $A$ and $B$ are mapped to the same place we obtain that images of red and green square curves are linked, i.e.\ the red curve in the image is passing through the hole of the green curve. Since $f_k$ converge to $f$ weakly in $W^{1,p}$ for $p\geq 1$, we obtain that either $f_k$ converge uniformly for $p>1$ or $Df_k$ are uniformly integrable for $p=1$ on those carefully chosen 1D green and red square curves and hence similar picture is true also for some $f_k$ for $k$ high enough. However, this gives us a contradiction as each homeomorphism preserves the linking number, i.e.\ it cannot map two curves that are not linked to two curves that are linked. 

\noindent
{\bf Acknowledgement.}
The authors would like to thank Anna Doležalová and Dalimil Peša for their help with the proof of Lemma \ref{LemmaTwoSets}.

\section{Preliminaries}
Let $n\in\N$. By $\left\lfloor \tfrac n2 \right\rfloor$ we denote the floor of $\tfrac n2$, that is 
$$
\left\lfloor \tfrac n2 \right\rfloor := \max\left\{m\in\Z : m\leq \tfrac n2\right\}.
$$
Let $a\in \R^n$. We denote by $Q(a,r)$ the cube with the center $a$ and the side length $r$, that is 
$$
Q(a,r):=(a_1-\tfrac{r}2,a_1+\tfrac{r}2)\times\cdots\times(a_n-\tfrac{r}2,a_n+\tfrac{r}2).
$$
By $B(a,r)$ we denote the open ball with the center $a$ and the radius $r$. We use $\Id$ for the identity matrix in $\R^n$. The Lebesgue measure in $\R^n$ will be denoted by $|\,\smallbullet\,|$. For two points $x=(x_1,\ldots,x_n)$ and $y=(y_1,\cdots,y_n)$ in $\R^n$ we define
$$\|x-y\|_\infty\:=\max\left\{|x_i-y_i|:i\in\{1,\ldots, n\}\right\}.$$ 

\begin{definition}\label{de:Lusin N}
Let $f\in W^{1, 1}_{\rm loc}(\Omega, \R^n)$ be a mapping. We say that $f$ satisfies the Lusin \LN condition on the measurable set $F\subset\Omega$ if for every set $E\subset F$ with $|E|=0$ we have $|f(E)|=0$.
\end{definition}

\begin{definition}\label{de:injae}
 Let $f\:\Omega\to \R^n$ be a mapping. We say that $f$ is injective almost everywhere on $\Omega$, if there exists a subset $N\subset\Omega$, $|N|=0$ such that for every $x_1, x_2\in\Omega\setminus N$ with $x_1\neq x_2$, we have $f(x_1)=f(x_2)$.
\end{definition}

For a mapping $f\in W^{1, 1}_{\rm loc}(\Omega,  \R^n)$, the ACL-characterization implies the differential matrix $Df(x)$ exists for almost every $x\in\Omega$. The following lemma can be found as a special case in \cite[Theorem 6.2]{EG2015}. 
\begin{lemma}\label{le:ApD}
Let $f\in W^{1,1}_{\rm loc}(\Omega, \R^n)$ be a mapping. Then for almost every $x\in\Omega$, we have 
\begin{equation}\label{eq:ApD}
\lim_{r\to0}\aint_{B(x, r)\cap\Omega}\left|\frac{f(y)-f(x)-Df(x)\cdot(y-x)}{r}\right|dy=0.
\end{equation}
\end{lemma}

\begin{definition}
Let $\Omega\subset \R^n$ be a domain and $f\:\overline\Omega\to \R^n$ be a continuous mapping which is smooth in~$\Omega$. Let $J_f(x)\neq 0$ for each $x\in\Omega\cap f^{-1}(y)$. Then we define the topological degree of $f$ by setting
$$
{\rm deg}(f, \Omega, y):=\sum_{x\in\Omega\cap f^{-1}(y)}{\rm sgn}(J_f(x)).
$$
With an approximation argument, this definition can be extended to an arbitrary continuous mapping from $\overline\Omega$ to $ \R^n$. 
\end{definition}
We should emphasize that the degree only depends on values of $f$ on $\partial\Omega$, see \cite{degree:book} for more details. For homeomorphism $f\:\overline\Omega\to \R^n$, either ${\rm deg} (f, \Omega, y)=1$ for all $y\in f(\Omega)$ (we say that $f$ is sense preserving), or ${\rm deg}(f, \Omega, y)=-1$ for all $y\in f(\Omega)$ (so $f$ is sense reversing).

\subsection{Linking number}
Let us consider two simple closed curves in $\R^3$. The linking number of this pair describes the number of times that each curve winds around the other (considering orientation), see Fig.\ \ref{fig:linkingnumber}. It is well known that the linking number is topologically invariant (see e.g.\ \cite{HM2010}). For the introductions to the linking number in $\R^3$ and its application to the theory of knots see \cite{R1976}, and for the applications in the theory of Sobolev mappings see \cite{HM2010} and~\cite{HO2016}.

We need only the fact that if you have two circles (or boundaries of squares) that are not linked (for example if the second circle does not intersect the disc whose boundary is the first circle), then their images by a homeomorphism also cannot be linked.

\begin{figure}[H]
	\begin{center}
		\includegraphics[page=4]{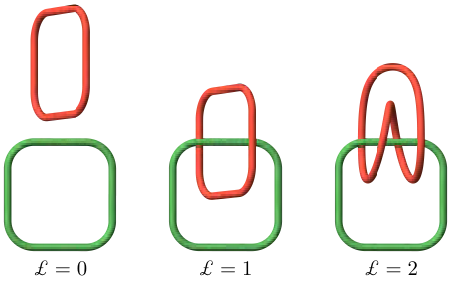}
	\end{center}
	\caption{Curves with linking numbers 0, 1 and 2}
	\label{fig:linkingnumber}
\end{figure}

The linking number can be defined even in the $n$-dimensional case. We now give the formal definition.
Let $n, p$ and $q$ be positive integers with $p+q=n-1$. Let us consider the 
mapping 
$\Phi(\xi,\eta)\colon
\overline{\mathbb B}_{p+1}\times\overline{\mathbb B}_{q+1}\to \R^n$
defined coordinate-wise as $\Phi(\xi,\eta)=x$, where
$$
\arraycolsep=0.7mm
\begin{array}{rcl}
	x_1&=&(2+\eta_1)\xi_1,\\
	&\vdots&\\
	x_{p+1}&=&(2+\eta_1)\xi_{p+1},\\
	x_{p+2}&=&\eta_2,\\
	&\vdots&\\
	x_{p+q+1}&=&\eta_{q+1}.
\end{array}
$$
Denote by $\mathbb A$ the anuloid 
$$\Phi(\mathbb S_p\times\mathbb B_{q+1})=\left\{x\in\R^n: \left(\sqrt{x_1^2+\cdots+x_{p+1}^2}-2\right)^2+x_{p+2}^2+\cdots+x_{n}^2< 1\right\} .$$ 
Of course given $x\in\overline{\mathbb A}$ we can find a unique $\xi\in \mathbb S_p$ and $\eta\in \overline{\mathbb B}_{q+1}$ such that 
$\Phi(\xi,\eta)=x$. We will denote these as $\xi(x)$ and $\eta(x)$. 

A \textit{link} is a pair $(\phi,\psi)$ of parametrized surfaces
$\phi\colon\mathbb S_p\to\R^n$, $\psi\colon\mathbb S_q\to\R^n$.
The \textit{linking number} of the link $(\phi,\psi)$ is 
defined as the topological degree 
$$
\lin (\phi,\psi)=\deg(L,\mathbb A,0),
$$
where the mapping
$L=L_{\phi,\psi}\:\overline{\mathbb A}\to \R^n$ is 
defined as $L(x)=\phi(\xi(x))-\bar\psi(-\eta(x))$, or equivalently
$$
 L(\Phi(\xi,\eta))=\phi(\xi)-\bar\psi(-\eta)\text{ for } \xi\in
\mathbb S_p, \eta\in\mathbb B_{q+1},
$$
where $\bar\psi$ is an arbitrary continuous extension of $\psi$ to 
$\overline{\mathbb B}_{q+1}$ (the degree does not depend on the way how we
extend $\psi$ as it depends only on the values on the boundary $\partial\mathbb A=\Phi(\mathbb S_p\times\mathbb S_q)$). 

\section{Proof of Theorem \ref{main}}
\begin{lemma}\label{LemmaTwoSets}
Let $f\in W^{1,1}_{\rm loc}(\Omega, \R^n)$ be a mapping with $J_f>0$ a.e.\ that is not injective almost everywhere, i.e. there is no set $N$ with $|N|=0$ so that $f$ is injective on $\Omega\setminus N$. Then there exist two measurable subsets $A, B\subset\Omega$ with $|A|>0$, $|B|>0$, $A\cap B=\emptyset$ and $f(A)=f(B)$.

Furthermore, if we take any $E\subset A$ (resp.\ $E\subset B$), $|E|>0$, then there exists a set $F\subset B$ (resp.\ $F\subset A$) such that $|F|>0$ and $f(E)=f(F)$.
\end{lemma}

\begin{proof}\ \\
\vspace{-3mm}
\begin{enumerate}
\item Let $\{Q_k\}_{k=1}^\infty$ be an open cover of $\Omega$, such that every $Q_k$ is bounded and for any two points $x\neq y$ we have a set $Q_i$ such that $y\notin Q_i\ni x$. Such a cover clearly exists as we can imagine for example the system of dyadic cubes of all sizes (together with slightly shifted dyadic cubes to also cover their boundaries).

\item We claim that for every $k\in\N$ there is a set $N_k$, $|N_k|=0$, such that 
$$f^{-1}(f(Q_k\setminus N_k))\text{ is measurable.}$$
\end{enumerate}
Indeed, let $k$ be fixed. Using the Lusin theorem, for every $m\in\N$, we can find a compact set $C_m\subset Q_k$ such that $f|_{C_m}$ is continuous and $|Q_k\setminus C_m|<\frac1m$. Define $N_k:=Q_k\setminus \bigcup_{m=1}^\infty C_m$. We have $|N_k|=0$. 
Clearly $\{f(C_m)\}$ are compact sets, and therefore 
$$
\bigcup_{m=1}^\infty f(C_m) = f\left(\bigcup_{m=1}^\infty C_m\right)=f(Q_k\setminus N_k)$$
is Borel ($F_\sigma$). Therefore the set $f^{-1}(f(Q_k\setminus N_k))$ is measurable as a preimage of Borel set under the measurable mapping.
\vspace{4mm}

\begin{enumerate}
\addtocounter{enumi}{2}
\item Since $f$ is Sobolev, there is a null set $N\subset \Omega$ such that $f$ satisfies the Lusin \LN condition on $\Omega\setminus N$, see e.g.\ \cite[Section A.8]{HK2014}. We may assume that $N\cap Q_k \subset N_k$, and so the Lusin \LN condition holds on $\Omega\setminus \bigcup_{k=1}^\infty N_k$.
\vspace{3mm}

 \item We claim that there is $k_0 \in \N$, such that 
 $$
 M_{k_0}:= f^{-1}(f(Q_{k_0}\setminus N_{k_0}))\setminus Q_{k_0} 
 \text{ satisfies }|M_{k_0}|>0.
 $$
\end{enumerate}
Let us assume for contradiction that $|M_k|=0$ for every $k$. Since $f$ is not injective a.e., we know that $f$ is not injective on the set of full measure 
$$\Omega \setminus \bigcup_{k=1}^\infty (N_k\cup M_k).$$
So there must be two different points $x\neq y$ in this set such that $f(x)=f(y)$. We find $i\in\N$ such that $x\in Q_i$ and $y\notin Q_i$. Since $x\notin N_i$, we know that $y\in f^{-1}(f(Q_i\setminus N_i))$. And from $y\notin Q_i$ it follows that $y\in M_i$, which is a contradiction with $y\in \Omega \setminus \bigcup_{k=1}^\infty (N_k\cup M_k)$.
\vspace{2mm}

\begin{enumerate}
\addtocounter{enumi}{4}
 \item The construction of the sets $A$ and $B$. 
\end{enumerate}
By making use of Lusin's theorem again, we can fix a compact set $A\subset M_{k_0}$ with ${|A|>\frac12|M_{k_0}|>0}$, such that $f$ is continuous on $A$, $f(A)$ is compact and
$f^{-1}(f(A))$ is measurable. We define 
$$B:=f^{-1}(f(A))\cap (Q_{k_0}\setminus N_{k_0}).$$

Clearly, $B\subset Q_{k_0}$, $A\cap Q_{k_0}=\emptyset$, and every point in $f(A)$ has at least two points in its preimage, one in $A$ and one in $B$, so $f(A) = f(B)$.

Since $|A|>0$ and $J_f>0$ a.e., $|f(A)|>0$. Because of step $(3)$, we know that $f$ satisfies the Lusin \LN condition on $B$. Therefore $|f(A)|=|f(B)|>0$ implies $|B|>0$. 

\begin{enumerate}
\addtocounter{enumi}{5}
 \item \qm{Furthermore} 
\end{enumerate}
Let $E\subset A$ be measurable with $|E|>0$ (the case $E\subset B$ is similar). Similarly to the previous step, we fix $F\subset B$ as 
$$F:=f^{-1}(f(E))\cap B.$$
As in step (2) we can find a set $N_E$, $|N_E|=0$ such that $f(E\setminus N_E)$ is Borel. Because of the Lusin \LN condition, we know that $|f(N_E)|=0$. It can be easily checked that $f(E)=f(F)$. And as before, using the positivity of $J_f$ a.e.\ and the Lusin \LN condition we can obtain that $|F|>0$.
\end{proof}

\begin{proof}[Proof of Theorem \ref{main} for $n=3$] 
For simplicity, we give the details of the proof only for $n=3$ and we explain how to proceed in general dimensions in the last step. 
The proof is divided into several steps. It is by contradiction, so let us assume that the statement does not hold.

\step{1}{Finding disjoint sets \texorpdfstring{$A$}{A} and \texorpdfstring{$B$}{B} that are mapped onto each other.}

By Lemma \ref{LemmaTwoSets} there exist two measurable subsets $A, B\subset\Omega$ such that
$$
A\cap B=\emptyset, |A|>0, |B|>0, |f(A)|>0 \text{ and } f(A)=f(B)
$$
and

\begin{equation}\label{eq:clubsuit}
\begin{aligned}
 &\text{for every } E\subset A, |E|>0, \text{ there is a set } F\subset B\text{ with } |F|>0\text{ and } f(E)=f(F),\\[1mm]
 &\text{for every } F\subset A, |F|>0, \text{ there is a set } E\subset B\text{ with } |E|>0\text{ and } f(E)=f(F).
\end{aligned}	
\end{equation}

\step{2a}{Making \texorpdfstring{$A$}{A} and \texorpdfstring{$B$}{B} smaller with better properties: approximate differentiability and Lebesgue points.}

Since $f\in W^{1, p}(\Omega, \R^n)$ for $p\geq 1$, $|Df(x)|<\infty$ for almost every $x\in A\cup B$. And we know that $f$ is almost everywhere approximately differentiable by Lemma \ref{de:injae}. So (by making the sets smaller, but maintaining their positive measures) we can without the loss of generality assume that every point $x\in A\cup B$ satisfies \eqref{eq:ApD}, $|Df(x)|<\infty$ and $J_f(x)>0$ (because $J_f>0$ a.e.) and every point $x\in A\cup B$ is a Lebesgue point of $Df$. 

\step{2b}{Making \texorpdfstring{$A$}{A} and \texorpdfstring{$B$}{B} smaller with better properties: almost constant derivative.}

We define $\mathcal M_{\Q}$ to be the class of all $3\times 3$ matrices with positive determinant and rational entries, i.e. 
\begin{equation}\label{eq:matrix}
\mathcal M_\mathbb Q:=\left\{M=(a_{i,j})_{i, j=1,2,3}: \det(M)>0,\ a_{i, j}\in\Q\right\}.
\end{equation}
Let $\delta_{\mathrm{2b}}>0$ be a sufficiently small constant to be fixed later. 
The Jacobian of $f$ is positive everywhere on $A$, so we can decompose $A$ as
 $$
 A=\bigcup_{M\in\mathcal M_\Q} \underbrace{\left\{x\in A:|Df(x)-M|\leq\delta_{\mathrm{2b}}|M|\right\}}_{=:A^M}.
 $$
 
As $|A|>0$, we know that at least one of those countable many measurable sets needs to have a positive measure and hence there exists a matrix $M_A\in\mathcal M_\mathbb Q$ such that $\left| A^{M_A}\right|>0$. Using \eqref{eq:clubsuit} we can find the corresponding $F$ such that $F\subset B$, $|F|>0$ and $f(A^{M_A}) = f(F)$. 

Therefore we can assume that for all $x\in A$ we have 
\begin{align}
|Df(x)-M_A|&\leq\delta_{\mathrm{2b}}|M_A|.\label{almostConstantDonA}\\
\intertext{We can do the same trick for $B$, so for every $x\in B$ we can have}
|Df(x)-M_B|&\leq\delta_{\mathrm{2b}}|M_B|.\label{almostConstantDonB}
\end{align}
Furthermore, we can assume that every point of $A$ and $B$ is a point of density for the respective set.

\step{2c}{Making \texorpdfstring{$A$}{A} and \texorpdfstring{$B$}{B} smaller with better properties and \qm{improving} \texorpdfstring{$f$}{f}: \texorpdfstring{$Df(x)\approx \Id$}{Df(x) ≈ Id}.}

Let us fix two points $a\in A$ and $b\in B$ such that $f(a)=f(b)$. We can without the loss of generality assume that $f(a)=f(b)=0$, $\dist(a,b)>2$ and $a=0$.

Let us define 
$$
r_a:=\frac{\min\left\{\dist(a, b), \dist(a, \partial\Omega)\right\}}
{3\max\left\{|M_A^{-1}|, 1\right\}}\text{ and }
r_b:=\frac{\min\left\{\dist(a, b), \dist(b, \partial\Omega)\right\}}
{3\max\left\{|M_B^{-1}|, 1\right\}}. 
$$
Since $a$ is a point of density of $A$, we know that $|A\cap B(a, r_a)|>0$. Similarly, since $b$ is a point of density of $B$, we know that $|B\cap B(b, r_b)|>0$. Therefore, using \eqref{eq:clubsuit}, we can assume that 
\begin{equation}
a\in A\subset B(a, r_a) \text{ and } b\in B\subset B(b, r_b).
\end{equation}
Note that $\dist(A,B) > \frac13 \dist(a,b) > 0$.

Let $H$ be a bi-Lipschitz diffeomorphism of $\overline{\Omega}$ onto itself such that $H$ is identity on $\partial \Omega$ and
$$
H(x)=
\begin{cases}
M_A^{-1}(x-a)+a&\text{for }x\in B(a, r_a),\\[2mm]
M_B^{-1}(x-b)+b&\text{for }x\in B(b, r_b).
\end{cases}
$$

The compositions $f\circ H$ and $f_k\circ H$ together with \eqref{almostConstantDonA} and \eqref{almostConstantDonB} allow us to assume without the loss of generality that for every $x\in A\cup B$ it holds that 
$$ |Df(x)-\Id|<\delta_{\mathrm{2b}}.$$
By possibly slightly modifying $M_A$ and $M_B$ in this argument, we may assume that $Df(a)=Df(b)=\Id$. Note that still $f(a)=f(b)=0$. 

We know that $f_k\wto f$ in $W^{1, p}(\Omega, \R^n)$, hence $f_k\to f$ in $L^p$, and we can assume that 
\begin{equation}
	f_k(x)\to f(x)\text{ for every }x\in A\cup B
\end{equation}
(a subsequence of $f_k$ converges pointwise almost everywhere, and we can make the sets $A$ and $B$ smaller).

Let us fix $\epsilon_{\mathrm{2c}}>0$. Since $f$ is approximately differentiable at $a$ and $b$ (it is app.\ diff.\ everywhere on $A\cup B$), by \eqref{eq:ApD} there exist $r_1>0$ such that for every $0<r\leq r_1$ we have
\begin{equation}\label{eq:AD1}
	\begin{aligned}
		\aint_{Q(a, r)}\left| f(x)- f(a)-D f(a)(x-a)\right|\dx&<r \cdot \epsilon_{\mathrm{2c}},\\
		\aint_{Q(b, r)}\left| f(x)- f(b)-D f(b)(x-b)\right|\dx&<r \cdot \epsilon_{\mathrm{2c}}.
	\end{aligned}
\end{equation}
Since $a$ and $b$ are Lebesgue points of $Df$, there exists $r_2>0$ sufficiently small such that for every $0<r<r_2$, we have (remember that $Df(a)=Df(b)=\Id$)
\begin{equation}\label{eq:DisId}
\begin{aligned}
	\aint_{ Q(a, r)}\left|Df(x)-\Id\right|\dx&<\epsilon_{\mathrm{2c}},\\
	\aint_{ Q(b, r)}\left|Df(x)-\Id\right|\dx&<\epsilon_{\mathrm{2c}}.
\end{aligned}
\end{equation}

The points $a$ and $b$ are the points of density of $A\cup B$ (or they can be if we make the sets smaller). Therefore there is $r_3>0$ such that for every $0<r\leq r_3$ it holds that 
\begin{equation}\label{eq:dens}
	\begin{aligned}
		\left| Q(a, r)\setminus A\right|\leq\epsilon_{\mathrm{2c}}\left| Q(a, r)\right|,\\
		\left| Q(b, r)\setminus B\right|\leq\epsilon_{\mathrm{2c}}\left| Q(b, r)\right|.
	\end{aligned}
\end{equation}

Let us fix $0<r<\min\{r_1, r_2, r_3, r_a, r_b\}$. There exists a bi-Lipschitz homeomorphism $R_1\:\R^3\to\R^3$ which maps $Q_a:=Q(a,1)$ to $Q(a,r)$ and $Q_b:=Q(b,1)$ to $Q(b,r)$, $\dist(Q_a,Q_b)>0$ and $DR_1(x)=r\cdot \Id$ for every $x\in Q(a,r)\cup Q(b,r)$.
And let us define the mapping 
$$R_2(x):=\frac{x}{r},$$
 so $DR_2(x)=\frac{1}{r}\Id$.

Now consider the compositions $R_2\circ f\circ R_1$ and $R_2\circ f_k \circ R_1$. It can be easily checked that all of the previous nice properties are preserved, and so we can assume that \eqref{eq:AD1} and \eqref{eq:dens} hold for $r=1$. In the case of \eqref{eq:AD1} we get (remember that $f(a)=f(b)=0$, $Df(a)=Df(b)=\Id$ and $|Q_a|=1$)
\begin{equation}\label{eq:AD2}
	\begin{aligned}
		\int_{ Q_a}\left| f(x)- (x-a)\right|\dx&< \epsilon_{\mathrm{2c}},\\
		\int_{ Q_b}\left| f(x)- (x-b)\right|\dx&<\epsilon_{\mathrm{2c}}.
	\end{aligned}
\end{equation}

\step{3a}{Construction of a closed curve in \texorpdfstring{$Q_a$}{Q\_a} where \texorpdfstring{$f$}{f} is almost identity.}

We define 
$$
U_A:=\left\{
	x\in Q_a:
	|f(x)-(x-a)|\geq \sqrt{\epsilon_{\mathrm{2c}}}
	\right\}.
	$$
Thanks to \eqref{eq:AD2} we have $\left|U_A\right|\leq \sqrt{\epsilon_{\mathrm{2c}}}.$ We know that $|Q_a\setminus A|\leq \epsilon_{\mathrm{2c}}$ (see \eqref{eq:dens}). By removing $U_A$ from $A$, we may without the loss of generality assume that $A\cap U_A=\emptyset$, $A\subset Q_a$ and 
\begin{equation}\label{eq:AisBig}
	|A|>1-(\sqrt{\epsilon_{\mathrm{2c}}}+\epsilon_{\mathrm{2c}}).
\end{equation}

Let us assume that $a=0$. We define the annular type set $\mathbf{L}_A\subset Q_a$ by setting (see Fig.~\ref{fig:GreenLinkTube}, and also Fig.\ \ref{fig:RedGreenLinkTubes})
$$
\mathbf{L}_A:=\left\{
	x=(x_1,x_2,x_3)\in Q_a:
	x_2\in\left(-\tfrac{1}{18}, \tfrac{1}{18}\right),
	\left\|(x_1,x_3)-(\tfrac{-2}{18},0)\right\|_{\infty}\in \left(\tfrac3{18},\tfrac5{18}\right)
	\right\}.
$$

\begin{figure}[H]
	\begin{center}
		\includegraphics[page=1]{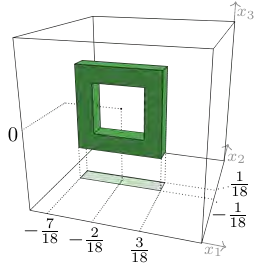}
	\end{center}
	\caption{The set $\mathbf{L}_A$ in $Q_a=Q(0,1)$.}
	\label{fig:GreenLinkTube}
\end{figure}

Then $\left|\mathbf{L}_A\right|=\frac{16}{729}>0$. 
For every $x_2\in\left(\frac{-1}{18}, \frac{1}{18}\right)$ and $r\in\left(\frac{3}{18}, \frac{5}{18}\right)$ we can define a closed simple curve $L_{x_2, r}\subset \mathbf{L}_A$ by setting 
$$
L_{x_2, r}:=
\left\{x=(x_1, x_2, x_3)\in Q_a :
\left\|(x_1,x_3)-(\tfrac{-2}{18},0)\right\|_{\infty} = r\right\}.
$$
Then $\mathcal H^1(L_{x_2, r})=8r\in\left(\frac{4}{3}, \frac{20}{9}\right)$ and $\mathbf{L}_A$ is a disjoint union of:
$$
\mathbf{L}_A=\bigcup_{\substack{x_2\in\left(\frac{-1}{18}, \frac{1}{18}\right),\\
 r\in\left(\frac{3}{18},\frac{5}{18}\right)}} L_{x_2, r}. 
$$

The ACL-characterization implies that for almost every $(x_2, r)\in\left(\frac{-1}{18}, \frac{1}{18}\right)\times\left(\frac{3}{18}, \frac{5}{18}\right)$ the restrictions of $f_k$ and $f$ onto the link $L_{x_2, r}$ belong to the Sobolev space $W^{1, p}(L_{x_2, r}, \R^n)$.
To simplify the notation, we still use $f_k$ and $f$ for the corresponding traces on the links~$L_{x_2, r}$.

Since $f_k\wto f$ in $W^{1, p}(Q_a)$, we may also assume that $f_{k}$ weakly converge to $f$ on $\mathbf{L}_A$ (where $f_k$ is a subsequence of the original sequence). Indeed, 
for a.e.\ $(x_2,r)$ there is a subsequence of $f_k$ which converges weakly to $f$ in $W^{1, p}(L_{x_2,r})$, see \cite[Lemma 2.9]{MS1995}. We therefore assume that on a.e.\ curve $f_k\wto f$ in $W^{1, p}(L_{x_2, r})$.
The inequality \eqref{eq:DisId} implies that for almost every $(x_2, r)\in\left(\frac{-1}{18}, \frac{1}{18}\right)\times\left(\frac{3}{18}, \frac{5}{18}\right)$ we have that 
\begin{equation}\label{eq:almostId}
\int_{L_{x_2, r}}\left|Df(x)-\Id\right|\dx<\epsilon_{\mathrm{2c}}.
\end{equation}
Because of \eqref{eq:AisBig}, the Fubini theorem implies that again for almost every $(x_2, r)$ we have 
\begin{equation}\label{eq:badIsSmall}
	\mathcal{H}^1(L_{x_2, r}\setminus A)<\epsilon_3,	
\end{equation}
where by taking $\epsilon_{\mathrm{2c}}$ sufficiently small, we can get an arbitrary small $\epsilon_3$.

\vspace{2mm}
We now fix $(x_2, r)$ which satisfies the previous conditions. From now on we assume that $p>1$, and at the end we explain how to fix the proof in the case $p=1$. Since $p>1$, we know that $W^{1,p}$ on the one-dimensional space on $L_{x_2, r}$ is compactly embedded into continuous functions and hence we can find a subsequence of $f_k$ (still denoted as $f_k$) that converges to $f$ uniformly on $L_{x_2, r}$. 
Let us fix $k_a\in\N$, such that for $k>k_a$ we have $|f_k(x)-f(x)|<\epsilon_{\mathrm{2c}}$ for all $x\in L_{x_2,r}$.

Let $x\in L_{x_2, r}\cap A$. Then $x\notin U_A$, and therefore 
\begin{equation}\label{eq:goodOnUA}
	|f_k(x)-x|\leq |f_k(x)-f(x)|+|f(x)-x|<\epsilon_{\mathrm{2c}} +\sqrt{\epsilon_{\mathrm{2c}}}.
\end{equation}

Let $x\in L_{x_2, r}\setminus A$. Since \eqref{eq:badIsSmall} holds, we know that there must be $y\in A\cap L_{x_2, r}$ such that $|x-y|<\frac{\epsilon_3}{2}$. Then we can estimate (since $f$ is ACL on $L_{x_2, r}$ and \eqref{eq:almostId})
\begin{equation}\label{eq:goodNotOnUA}
	\begin{aligned}
	|f_k(x)-x|&\leq |f_k(x)-f(x)|+|f(x)-f(y)|+|f(y)-y|+|y-x| \\
	&\leq \epsilon_{\mathrm{2c}}
	+ \Bigl(\int_{L_{x_2, r}\setminus A} |Df(t)|\dt\Bigr)
	 + \sqrt{\epsilon_{\mathrm{2c}}} 
	 +\epsilon_3
	 \\
	&\leq \epsilon_{\mathrm{2c}} + \sqrt{\epsilon_{\mathrm{2c}}} + \epsilon_3 +
	\int_{L_{x_2, r}\setminus A} \bigl(|Df(t)-\Id|+|\Id|\bigr)\dt \\
	&\leq \epsilon_{\mathrm{2c}} + \sqrt{\epsilon_{\mathrm{2c}}} +\epsilon_3+ \epsilon_{\mathrm{2c}} + \epsilon_3.
\end{aligned}
\end{equation}

The above estimates \eqref{eq:goodOnUA} and \eqref{eq:goodNotOnUA} imply that for arbitrary $\epsilon>0$ there exists a large enough $k_a\in \N$ such that for every $k>k_a$ and every $x\in L_{x_2, r}$, we have
\begin{equation}\label{eq:allIsGoodOnA}
|f_k(x)-x|<\epsilon.
\end{equation}

\step{3b}{Construction of a closed curve in \texorpdfstring{$Q_b$}{Q\_b} where \texorpdfstring{$f$}{f} is almost translation.}

Just as in the case of $Q_a$ we can define the set $\mathbf{L}_B$ for $Q_b$, see Fig.\ \ref{fig:RedLinkTube} and Fig.\ \ref{fig:RedGreenLinkTubes},
$$
\mathbf{L}_B:=\left\{
	x=(x_1,x_2,x_3)\in Q_a:
	x_3\in\left(-\tfrac{1}{18}, \tfrac{1}{18}\right),
	\left\|(x_1,x_2)-\left(\tfrac2{18},0\right)\right\|_{\infty}\in \left(\tfrac{3}{18}, \tfrac{5}{18}\right)
	\right\}.
$$
\begin{figure}[H]
	\begin{center}
		\includegraphics[page=2]{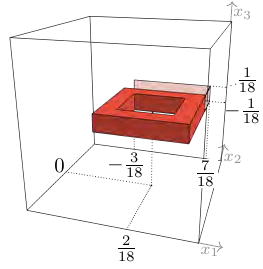}
	\end{center}
	\caption{The set $\mathbf{L}_B$ in $Q_b$, the coordinates are relative to the point $b$.}
	\label{fig:RedLinkTube}
\end{figure}

And as in the previous step, we can find a curve 
$\widehat{L}_{x_3,r}\subset \mathbf{L}_B$,
$$\widehat{L}_{x_3,r}:=
\left\{x=(x_1, x_2, x_3)\in Q_a :
\left\|(x_1,x_2)-(\tfrac2{18},0)\right\|_{\infty} = r\right\},
$$
such that for any $k>k_b$ and $x\in \widehat{L}_{x_3,r}$ we have 
\begin{equation}\label{eq:allIsGoodOnB}
|f_k(x)-x-b|<\epsilon.
\end{equation}

\step{4}{The contradiction.}

Since $Q_a$ and $Q_b$ are far away from each other, the curves $L_{x_2, r}$ and $\widehat{L}_{x_3,r}$ are not linked. 
If we move the sets $\mathbf{L}_A$ and $\mathbf{L}_B$ to the same cube $Q_a=Q(0,1)$, see Fig.\ \ref{fig:RedGreenLinkTubes}, we can see that the curves $L_{x_2, r}$ and $\widehat{L}_{x_3, r}-b$ are linked for any choce of $x_2$, $x_3$ and $r$. 
\begin{figure}[H]
	\begin{center}
		\includegraphics[page=3]{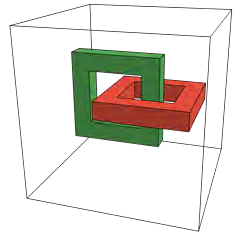}
	\end{center}
	\caption{The sets $\mathbf{L}_A$ and $\mathbf{L}_B$ moved to the same cube $Q(0,1)$.}
	\label{fig:RedGreenLinkTubes}
\end{figure}

Because of \eqref{eq:allIsGoodOnA} and \eqref{eq:allIsGoodOnB}, we can clearly choose $\epsilon>0$ such that the curves $f_k(L_{x_2, r})$ and $f_k(\widehat{L}_{x_3,r})$ are still linked for $k>\max\{k_a,k_b\}$, see Fig.\ \ref{fig:LinkedSquares}.
Since $f_k$ is a homeomorphism, the images of not-linked curves must not be linked, which gives us a contradiction.
\begin{figure}[H]
	\begin{center}
		\includegraphics[height=8.25cm, page=5]{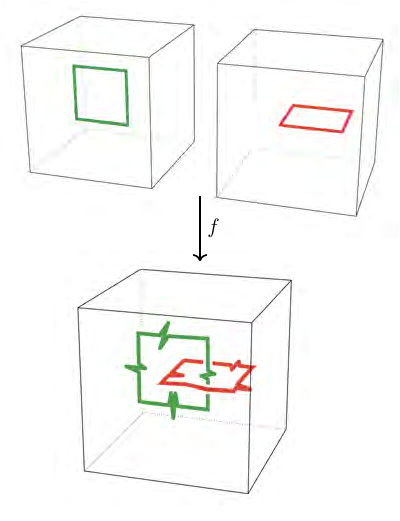}
	\end{center}
	\caption{The images of the curves $L_{x_2, r}$ and $\widehat{L}_{x_3,r}$ are linked.}
	\label{fig:LinkedSquares}
\end{figure}

\step{5}{The missing case \texorpdfstring{$p=1$}{p=1}.} In the case $p=1$ we have to slightly adjust our argument, as we may not have uniform convergence of $f_k$ on $L_{x_2,r}$. 
We know that $Df_k$ converge weakly to $Df$ in $L^1$, so by the de la Vallée Poussin characterization of weak convergence we obtain that $Df_k$ are uniformly integrable, i.e. 
\begin{equation}\label{defphi}
\text{ for any } \epsilon>0\, \text{ there is } \delta>0\, \text{ so that for any } k\in\N\:|A|<\delta\Rightarrow \int_A |Df_k|<\epsilon. 
\end{equation}

As in the previous part of the proof, we know that $f$ behaves on $L_{x_2,r}$ like in Fig \ref{fig:LinkedSquares}, so the mapping $f$ is really close to the identity on $L_{x_2,r}$. We fix $\delta>0$ and we study the \qm{ugly} set
$$
U:=\{x\in\Omega:\ |f_k(x)-f(x)|\geq \tfrac{1}{1000}\}.
$$
Since $f_k$ converge to $f$ pointwise a.e., we can find $k_0$ big enough so that for $k\geq k_0$ we have 
$
|\mathbf{L}_A\cap U|<\delta\, |\mathbf{L}_A|.
$

Now there are two possibilities. If
\begin{equation}\label{key}
\int_{U\cap L_{x_2,r}} |Df_k|<\frac{1}{1000},
\end{equation}
then we know that $f_k$ is close to $f$ on $L_{x_2,r}\setminus U$ by the definition of $U$. By \eqref{key} it cannot go far away on $U\cap L_{x_2,r}$ so it is close to $f$ on the whole $L_{x_2,r}$. 
It follows that the situation is as in Fig.\ \ref{fig:LinkedSquares} also for $f_k$, i.e.\ $f_k$ is really close to the identity mapping on $L_{x_2,r}$. If the same thing happens also on $\widehat{L}_{x_3,r}$, i.e.\ if the analogy of \eqref{key} holds also there, we get that $f_k$ is close to $f$ also on $\widehat{L}_{x_3,r}$ and we obtain the desired contradiction with the stability of linking in the same way as in previous step -- see Fig.\ \ref{fig:LinkedSquares}. 

So we may assume that \eqref{key} fails either for $L_{x_2,r}$ or for $\widehat{L}_{x_3,r}$. Without the loss of generality, we assume that it fails for $L_{x_2,r}$. And we may assume that it fails for any square $L_{x_2,r}\subset \mathbf{L}_A$ where our other conditions are satisfied and those squares occupy a big portion of $\mathbf{L}_A$, so we obtain that 
$$
\int_{U\cap \mathbf{L}_A} |Df_k|\geq C\text{\ \ but\ \ }
|\mathbf{L}_A\cap U|<\delta |\mathbf{L}_A|,
$$
which gives us a contradition with uniform integrability of $|Df_k|$ once we choose $\delta$ sufficiently small. 

\step{6}{Proof for the case when \texorpdfstring{$n=2$}{n=2} or \texorpdfstring{$n>3$}{n>3}.}

When $n=2$, De Philippis and Pratelli \cite[Lemma 2.6]{DPP2020} proved that if $\{f_k\}\subset W^{1,p}(\Omega,  \R^2)$ is a sequence of mappings which satisfy the \INV condition for $1\leq p<\infty$ and they converge weakly to $f\in W^{1, p}(\Omega,  \R^2)$ with $J_f>0$ almost everywhere, then $f$ also satisfies the \INV condition. Since homeomorphisms must satisfy the \INV condition, the Theorem~\ref{main} directly follows from their result.

When $n>3$, we can use a similar idea as in the proof for $n=3$. The only difference is the definition of the surfaces $L_{x_2, r}$ and $\widehat{L}_{x_3,r}$. In the case $n=3$, we had two curves -- one-dimensional surfaces. In the case of $n>3$, $n$ odd we need two $(\frac{n-1}2)$-dimensional surfaces. In the case of $n$ even we need one $(\frac{n}2-1)$-dimensional surface and one $(\frac{n}2)$-dimensional surface. The rest of the proof is the same, since in either case, we have that the dimension of either surface is not higher than $\left\lfloor \frac{n}{2}\right\rfloor<p$, and therefore $W^{1,p}$ is continuously embedded into continuous functions on those sets and we can use the same argument as in the case $n=3$.
\end{proof}

\bibliographystyle{plain}
\bibliography{Bibliography} 
\newpage
\end{document}